\documentclass[11pt]{article}

\usepackage{amssymb,amsmath,amsfonts,amsthm}
\usepackage{graphicx,color,enumitem}
\usepackage{bm}
\usepackage{mathtools}
\usepackage{tikz,tikz-cd}

\usepackage[round]{natbib}

\usepackage{geometry}

\usepackage[colorlinks=true,citecolor=blue,pagebackref=true]{hyperref}

\newcommand{\R}{{\mathbb R}}

\newcommand{\N}{{\mathbb N}}

\newcommand{\Lcal}{{\mathcal L}}
\newcommand{\Ucal}{{\mathcal U}}
\newcommand{\Mcal}{{\mathcal M}}
\newcommand{\Ncal}{{\mathcal N}}

\newcommand{\Vcal}{{\mathcal V}}

\newcommand{\fdot}{{\,\cdot\,}}

\newtheorem{theorem}{Theorem}

\newtheorem{corollary}[theorem]{Corollary}
\newtheorem{definition}[theorem]{Definition}

\newtheorem{lemma}[theorem]{Lemma}

\newtheorem{proposition}[theorem]{Proposition}
\newtheorem{remark}[theorem]{Remark}

\theoremstyle{definition}
\newtheorem{example}[theorem]{Example}

\numberwithin{equation}{section}
\numberwithin{theorem}{section}

\definecolor{darkgreen}{rgb}{0,0.7,0}

\newcommand{\iii}{{\vert\kern-0.25ex\vert\kern-0.25ex\vert}}

\begin{document}

\title{Deep neural networks, generic universal interpolation, \\and controlled ODEs\thanks{The authors thank the Associate Editor and two anonymous referees for their valuable comments. Christa Cuchiero gratefully acknowledges financial support by the Vienna Science and Technology Fund (WWTF) under grant MA16-021. Josef Teichmann gratefully acknowledge financial support by the Swiss National Science Foundation (SNF) under grant 179114.}}
\author{Christa Cuchiero\thanks{Department of Statistics and Operations Research, Data Science @ Uni Vienna, University of Vienna (christa.cuchiero@univie.ac.at).}
\and Martin Larsson\thanks{Department of Mathematical Sciences, Carnegie Mellon University (martinl@andrew.cmu.edu).}
\and Josef Teichmann\thanks{Department of Mathematics, ETH Zurich (jteichma@math.ethz.ch).}}


\maketitle

\begin{abstract}
  A recent paradigm views deep neural networks as discretizations of
  certain controlled ordinary differential equations, sometimes called neural ordinary differential equations. We make use of
  this perspective to link expressiveness of deep networks to the
  notion of controllability of dynamical systems. Using this
  connection, we study an expressiveness property that we call
  \emph{universal interpolation}, and show that it is generic in a
  certain sense. The universal interpolation property is slightly
  weaker than universal approximation, and disentangles supervised
  learning on finite training sets from generalization properties. We
  also show that universal interpolation holds for certain deep neural
  networks even if large numbers of parameters are left untrained, and
  are instead chosen randomly.  This lends theoretical support to the
  observation that training with random initialization can be
  successful even when most parameters are largely unchanged through
  the training. Our results also explore what a minimal amount of trainable parameters in neural ordinary differential equations could be without giving up on expressiveness.
\end{abstract}

\section{Deep neural networks as controlled ODEs}

Several recent studies of deep neural networks revolve around the idea
of viewing such networks as discretizations of ordinary differential
equations (ODEs). This led to the terminology \emph{neural ODEs}, a perspective which has successfully been applied to a
number problems; see e.g.\
\cite{e_17,chang_etal_17,che_etal_18,gra_che_bet_sut_duv_18,dup_dou_teh_19}
among many others. See also \cite{e_han_li_18,liu_mar_19} for
mathematically rigorous analyses. In this paper we make progress
towards a theoretical understanding of this success. Using ideas from
dynamical systems and control theory, we show why it can be beneficial
to view deep neural networks as discretized controlled ODEs. Our
analysis suggests that randomization of the vector fields can be used
to substantially reduce the number of trainable parameters. This sheds
new light on random initialization of deep neural networks with fully
trainable parameters.

The approach in \cite{e_17,che_etal_18,liu_mar_19} rests on the
observation that the input $X_k$ to any given layer $k$ is mapped to
an output $X_{k+1}$ that can be expressed as a residual network style
transition \citep{he_zhang_ren_sun_15} of the form
$X_{k+1} = X_k + V(X_k,\theta_k)$. The right-hand side depends both on
the input $X_k$ and on a parameter vector $\theta_k$, both of which
vary from layer to layer.

The representation of $X_{k+1}$ as a perturbation of $X_k$ suggests
that for sufficiently deep networks, the cumulative effect of repeated
transitions mimics the behavior of an ODE. This ODE can then be
studied instead of the original network. The discrete parameter
$k=0,1,2,\ldots$ that counts the layers is replaced by a continuous
parameter $t\in[0,1]$, and one lets the ``state'' $X_t$ at ``layer''
$t$ evolve according to a law of motion of the form
\begin{equation}\label{CODE_1}
  \frac{d}{dt}X_t = V(X_t,\theta_t).
\end{equation}
In other words, one views depth as the running time of a dynamical
system. The solution $X_t$ of \eqref{CODE_1} forms a curve through its
state space, which we here take to be $\R^m$ for some fixed dimension
$m$, and $\theta_t$ represents a curve through the space of possible
parameters. Given an initial condition $x\in\R^m$, we let $X^x_t$
denote the corresponding solution of \eqref{CODE_1}, subject to
\[
  X^x_0 = x.
\]
For all choices of $V(x,\theta)$ and $\theta_t$ considered in this
paper, the solution of \eqref{CODE_1} exists and is unique.  The
following example connects \eqref{CODE_1} with standard neural network
architectures.

\begin{example}\label{Ex_1}
  In a standard (residual) neural network layer, the components of
  $V(x,\theta)$ are of the form
  $V^j(x,\theta)=b^j + \sum_{k=1}^m a^j_k\sigma(x^k)$ for
  $j=1,\ldots,m$, where the parameters $a^j_k$, $b^j$ make up the
  vector $\theta$, and $\sigma(\fdot)$ is a fixed nonlinearity acting
  on the components of $x=(x^1,\ldots,x^m)$.  In this example,
  somewhat oddly, we let the nonlinearity act \emph{before} the affine
  map. However, the ordering is inessential when multiple layers are
  composed, because the nonlinearity takes the affine map from the
  previous layer as input. The choice made here will be convenient in
  later examples.
\end{example}

For an input $x\in\R^m$, the ``continuous-depth'' network
\eqref{CODE_1} outputs $X^x_1$. This is however still a vector in
$\R^m$, and will usually be mapped to a much lower dimensional output,
say $R(X^x_1)$ for some readout map $R\colon\R^m\to\R^{m'}$ with
$m'\ll m$. Supervised learning in this framework amounts to the
following: for a given training set of input/ouput pairs,
$(x_i,y_i)\in\R^m\times\R^{m'}$ for $i=1,\ldots,N$, identify
parameters $\theta_t$, $t\in[0,1]$, and a readout map $R$ such that
$R(X^{x_i}_1) \approx y_i$ for all $i$, perhaps while imposing a
regularization penalty on $\theta_t$. Our results are formulated for
$m'=m$ with either the identity readout, leading to $x\mapsto X^x_1$,
or the readout structure $x\mapsto \lambda(X^x_1-x)$ that depends
directly on the input data and a trained scalar parameter $\lambda>0$.

In the present paper we recognize \eqref{CODE_1} as a \emph{controlled
  ordinary differential equation} (CODE), and the training task as a
problem of optimal control. One of our key motivations is to show that it is actually not necessary to train all parameters. Only a minority needs to be trained. We capture this idea by
decomposing $V(x,\theta)$ in a way where the dependence on the trainable parameters enters linearly, which corresponds to the most natural and  simplest parametrization. 
Indeed, our results will be proved in the
following setting.
Suppose the function $V(x,\theta)$ determining the
right-hand side of \eqref{CODE_1} is of the form
\begin{equation}\label{u1V1udVd}
  V(x,\theta) = u^1V_1(x) + \cdots + u^dV_d(x),
\end{equation}
where $u^1,\ldots,u^d$ are scalar parameters, and $V_1,\ldots,V_d$ are
smooth vector fields on $\R^m$.\footnote{That is, the $V_i$ are smooth
  maps from $\R^m$ to $\R^m$.} We think of $u^1,\ldots,u^d$ as
trainable parameters (thus part of $\theta$) that will be
$t$-dependent. The vector fields $V_1,\ldots,V_d$ are specified by the
remaining parameters in $\theta$, which will be non-trainable and
constant in $t$. The following example illustrates that this decomposition is in line with the standard neural network
architecture of Example \ref{Ex_1}.

\begin{example}\label{Ex_2}
  Recall the standard architecture of Example \ref{Ex_1}, where each
  layer depends on $m+ m^2$ parameters.  If each vector field $V_i(x)$
  is of this form, then so is $V(x,\theta)$ in \eqref{u1V1udVd}. To
  see this, suppose
  $V_i^j(x)=b_i^j + \sum_{k=1}^m a^j_{ik}\sigma(x^k)$ for some
  parameters $b_i^j,a^j_{ik}$. Then
  $V^j(x,\theta)=b^j + \sum_{k=1}^m a^j_k\sigma(x^k)$ with
  $b^j=\sum_{i=1}^du^ib_i^j$ and $a^j_k=\sum_{i=1}^d u^i a^j_{ik}$,
  which again has the standard form in Example~\ref{Ex_1}. This
  construction should be viewed as one way of decomposing the full
  parameter set into trainable and non-trainable parameters. In fact,
  in this example, the number of trainable parameters per layer is
  $d$, which should be thought of as being much smaller than the
  number of non-trainable parameters $m+m^2$. A key message of our
  results is that similar reductions in the number of non-trainable
  parameters are possible in the CODE setting, without compromising
  expressive power.
\end{example}

With the specification \eqref{u1V1udVd}, the CODE \eqref{CODE_1} takes
the form
\begin{equation}\label{CODE_2}
  \frac{d}{dt}X_t = u^1_tV_1(X_t) + \cdots + u^d_tV_d(X_t),
\end{equation}
where $u^1_t,\ldots,u^d_t$ are the controls (the trainable
parameters). As before, if the initial condition is $x$, the solution
is denoted by $X^x_t$. The output is $X^x_1$, or if composed with a
readout, $R(X^x_1)$. If the controls are square-integrable functions
of $t$ and the vector fields are smooth and bounded (i.e.,
$\sup_{x\in\R^m}\|V_i(x)\| < \infty$ for all $i$), one has existence
and uniqueness of solutions of \eqref{CODE_2} for every initial
condition.

The system \eqref{CODE_2} turns out to be remarkably expressive if the
vector fields are chosen appropriately. Our goal in this paper is to
make this statement rigorous. In Section~\ref{S_gen_int} we establish
Theorem~\ref{T1}, which states that one can match any training set of
finite size using just $d=5$ suitably chosen vector fields
$V_1,\ldots,V_5$. That is, for any finite set of input/ouput pairs
$(x_i, y_i)\in\R^m\times\R^{m}$, there exist controls such that
$X^{x_i}_1=y_i$ for all $i$. We refer to this property as the
\emph{universal interpolation property}. This differs from the
well-studied notion of universal approximation (e.g.\
\cite{cyb_89,hor_91}), and makes no statement about generalization
properties. Let us stress that we do not claim that perfect
interpolation is necessarily a desirable training goal. Still, we
believe it serves as a useful measure of expressiveness. Moreover, recent work on the so-called \emph{double-descent phenomenon} has shown that even when machine learning models interpolate the training set, they can still generalize well on unseen data; see e.g.\ \cite{Ma2018ThePO,Belkin2018OverfittingOP,liang2018just}. For classical
results on interpolation via neural networks, e.g.~multilayer
feedforward perceptrons, we refer to \citet[Theorem
5.1]{pinkus_99}. In contrast to this classical theorem, our result
does not depend on the number of training samples that one aims to
match. Recently, universal approximation of neural ordinary differential
equations has been considered in \cite{ZGUA:2019}, where the authors
prove that certain homeomorphism on $\mathbb{R}^p$ can be embedded
into flows of controlled ordinary differential equations on
$\mathbb{R}^{2p}$. One essential difference to our results is the question of minimal controllability of the flows, which is not addressed in \cite{ZGUA:2019}. No-go
results have been shown in \cite{DDT:2019}. These results do not
contradict our findings as we only work with finite training data
sets.

The proofs of our results rely on mathematical machinery from control
theory, involving classical notions like Lie brackets and
controllability. This is reviewed in Section~\ref{S_Lie}. In addition
to laying the groundwork for the proofs, we aim to convey the
intuition for why control theory can help explain expressiveness in
deep learning. The formal proof of Theorem~\ref{T1} is then given in
Sections~\ref{S_proof_T1}, with some lengthier computations postponed
to the Appendix.

In Section~\ref{S_genexp} we go further by showing that not only are
five vector fields enough, they can be chosen randomly in the class of
real analytic vector fields. We make this precise in Theorem~\ref{T2}.
As a consequence, common structures such as the one in
Example~\ref{Ex_1} (with real analytic nonlinearities such as the standard functions $\arctan(x)$ or $\tanh(x)$) can be shown to
retain this strong form of expressiveness. This is done in
Corollary~\ref{C2}.

We do not make any statement about optimality of these generic
expressive networks for specific learning tasks. However, our analysis
produces the remarkable conclusion that deep neural networks,
expressed as discretizations of \eqref{CODE_2} with only five random
vector fields, can interpolate any functional relation with a
precision that depends only on depth and the amount of training
data. Our approach supports the ``folklore'' statement that randomness
is of great importance for training. Indeed, the role of randomness,
which is ubiquitous in training procedures (\emph{stochastic} gradient
descent, \emph{random} initialization of weights, etc.), receives a
theoretical basis through Theorem~\ref{T2}.  In Section
\ref{S_genexp}, we comment on these algorithmic aspects, although we
do not perform any empirical analysis in this paper.  The proof of
Theorem~\ref{T2} is given in Section~\ref{S_proof_T2}.

A full-fledged geometric and quantitative analysis in a very general
analytic setting is performed in the companion
paper~\cite{cuclartei:19}. There $ \mathbb{R}^m $ is replaced by a
so-called convenient vector space, covering various
infinite-dimensional situations of interest. We give a new proof of
the Chow--Rashevskii theorem, and present quantitative results on
training controlled ODEs. This lets us analyze controlled transport
equations or PDEs, as well as the effect of convolutional layers.

\section{Universal interpolation}\label{S_gen_int}

Interpreting \eqref{CODE_1} and \eqref{CODE_2} as CODEs establishes an
interface to control theory. This opens the door to powerful
mathematical techniques that we will deploy to establish an
expressiveness property that we call \emph{universal
  interpolation}. When satisfied, this property guarantees that any
supervised learning task has a solution. It is formalized in the
following definition, which uses the identity readout $R(x)=x$.

\begin{definition}
  The control system \eqref{CODE_2}, specified by the vector fields
  $V_1,\ldots,V_d$, is called a \emph{universal $N$-point
    interpolator} on a subset $\Omega\subseteq\R^m$ if, for any
  training set $\{(x_i,y_i)\in\Omega\times\Omega\colon i=1,\ldots,N\}$
  of size $N$, there exist controls $u^1_t,\ldots,u^d_t$ that achieve
  the exact matching $X^{x_i}_1=y_i$ for all $i=1,\ldots,N$. Here it
  is required that the training inputs $x_1,\ldots,x_N$, as well as
  the targets $y_1,\ldots,y_N$, are pairwise distinct.\footnote{A
    system like \eqref{CODE_2} can never map different inputs to the
    same output. Moreover, it is not meaningful to pair one single
    input with two different outputs in the training set.}
\end{definition}

Universal $N$-point interpolation may look like a rather strong
requirement, especially if the size $N$ of the training set and/or the
ambient dimension $m$ is large. Clearly, this property is primarily of
interest if the number $d$ of vector fields can be chosen small
compared to $N$ and $m$. Our first main result states that, in a
striking manner, this is always possible.

\begin{theorem}\label{T1}
  Fix $m\ge2$ and a bounded open connected subset
  $\Omega\subset\R^m$. There exist $d=5$ smooth bounded vector fields
  $V_1,\ldots,V_5$ on $\R^m$ such that \eqref{CODE_2} is a universal
  $N$-point interpolator in $\Omega$, for every $N$.
\end{theorem}

The formal proof of Theorem~\ref{T1} is presented in
Section~\ref{S_proof_T1}, building on classical ideas from control
theory reviewed in Section~\ref{S_Lie}. Before discussing the proof,
let us comment on the content of the theorem.

First, observe that $V_1,\ldots,V_5$ do not depend on $N$. Thus the
\emph{same} five vector fields can be used to interpolate any
arbitrary (but finite) training set. Of course, the controls
$u^1_t,\ldots,u^d_t$ that achieve interpolation do depend on the
training set. If the training set changes, for example if it is
augmented with additional training pairs, the controls will generally
change as well.

Next, the vector fields themselves depend on the ambient dimension
$m$, by the very definition of a vector field on $\R^m$. However, we
stress that no matter how large $m$ is, $d=5$ vector fields always
suffice to achieve universal interpolation for arbitrarily large
training sets.

Further, the case $m=1$ is not covered. This reflects the fact that
$N$ points $x_1,\ldots,x_N$ on the real line cannot be continuously
transported to targets $y_1,\ldots,y_N$ without intersecting, if the
inputs and targets are ordered differently. Such a training task
cannot be achieved by \eqref{CODE_2}, since trajectories
$\{X_t\colon t\in[0,1]\}$ corresponding to different initial
conditions always remain disjoint.

Finally, Theorem~\ref{T1} is an existence result with no quantitative
estimates on, for example, the size of the controls
$u^1_t,\ldots,u^d_t$ needed to achieve interpolation. Similarly,
nothing is asserted regarding the behavior of the map $x\mapsto X^x_1$
away from the training inputs $x_i$. In practice, one does not insist
on exact interpolation, but trades off accuracy for more regular
controls. A rigorous analysis of these issues would be of great
interest, though it is not the subject of this paper. Here we only
provide the following proposition which states the form of the first
derivative of $x\mapsto X^x_1$, along with a bound on its size in
terms of the size of the vector fields and controls. The derivative of
$x\mapsto X^x_t$ at a point $x$ (called \emph{first variation}) is an
$m\times m$ matrix that we denote by $J_t^x$ for Jacobian.

\begin{proposition}
  Consider the CODE \eqref{CODE_2} under the assumptions of existence
  and uniqueness. Then $J_t^x$ solves the linear differential equation
  \begin{equation}\label{first_variation}
    \frac{d}{dt} J_{t}^x  =   \sum_{i=1}^d   u^i_t  D V_i(X^x_t) J_{t}^x,  \quad t \in [0,1],
  \end{equation}
  with initial value $J_{0}^x=I$ (the $m\times m$ identity matrix),
  where $DV_i$ denotes the Jacobian of the vector field $V_i$. The
  operator norm of $J^x_t$ is bounded by
  \[
    \| J^x_t \|_\text{op} \le \exp\left( \int_0^t \| \sum_{i=1}^d
      u^i_s\, DV_i(X^x_s)\|_\text{op} \, ds\right).
  \]
\end{proposition}

\begin{proof}
  We obtain \eqref{first_variation} by differentiating the equation
  $X_t^x = x + \sum_{i=1}^d \int_0^t u^i_sV_i(X^x_s) \, ds$ and
  applying the chain rule. To deduce the bound on
  $\|J^x_t\|_\text{op}$, pick an arbitrary unit vector $z\in\R^d$ and
  use \eqref{first_variation} along with the triangle inequality and
  the definition of the operator norm to get
  \[
    \| J^x_t z\| \le 1 + \int_0^t \| \sum_{i=1}^d u^i_s\,
    DV_i(X^x_s)\|_\text{op} \|J^x_s z\| ds.
  \]
  Gronwall's inequality yields
  $\| J^x_t z \| \le \exp( \int_0^t \sum_{i=1}^d |u^i_s|\,
  \|DV_i(X^x_s)\|_\text{op} ds)$. This implies the claimed bound on
  $\|J^x_t\|_\text{op}$ since $z$ was an arbitrary unit vector.
\end{proof}

Some related quantitative questions are discussed in the companion
paper \cite{cuclartei:19}.

\section{Lie brackets and controllability}\label{S_Lie}

In preparation for the proof of Theorem~\ref{T1}, and to aid intuition
as to why such a small number of vector fields can result in a highly
expressive system, we review some ideas from control theory. The
developments take place in a generic Euclidean space $\R^n$; later we
will take $n=mN$, where $N$ is the size of the training set. As we do
not assume the reader is familiar with this theory, we will give
examples in an attempt to convey the underlying intuition.

\begin{definition}
  Let $U$, $V$, $U_1,\ldots,U_d$ be smooth vector fields on $\R^n$.
  \begin{itemize}
  \item The \emph{Lie bracket} $[U,V]$ is the smooth vector field on
    $\R^n$ given by
    \[ [U,V](x) = DV(x) \, U(x) - DU(x) \,V(x),
    \]
    where $DU$ is the Jacobian matrix of partial derivatives; thus its
    $(i,j)$ entry is $\partial U^i / \partial x^j$, and similarly for
    $DV(x)$.
  \item The \emph{Lie algebra generated by $U_1,\ldots,U_d$}, denoted
    by ${\rm Lie}(U_1,\ldots,U_d)$, is the smallest linear space of
    vector fields that contains $U_1,\ldots,U_d$ and is stable under
    Lie brackets. Equivalently, we have
    \[ {\rm Lie}(U_1,\ldots,U_d) = {\rm span}\{\text{$U_1,\ldots,U_d$
        and all iterated Lie brackets}\}.
    \]
    For any $x\in\R^n$, we also consider the subspace of $\R^n$
    obtained by evaluating all the vector fields in the Lie algebra at
    $x$, namely
    \[ {\rm Lie}(U_1,\ldots,U_d)(x) = \{W(x)\colon W\in {\rm
        Lie}(U_1,\ldots,U_d)\} \subseteq \R^n.
    \]
  \end{itemize}
\end{definition}

Let us look at the case of linear vector fields, where the Lie
brackets have simple expressions.

\begin{example}\label{Ex_Lie_linear}
  Consider linear vector fields $U(x)=Ax$ and $V(x)=Bx$, where $A$ and
  $B$ are $n\times n$ matrices. A direct calculation shows that
  $[U,V](x)=(AB-BA)x$. Therefore, ${\rm Lie}(U,V)$ consists of all
  linear vector fields of the form $Cx$, where $C$ is obtained from
  $A$ and $B$ by taking matrix commutators and linear combinations
  finitely many times.
\end{example}

The main tool in the proof of Theorem~\ref{T1} is the Chow--Rashevskii
theorem, which can be stated as follows. For details, see
\cite[Chapter~2]{mon_02}.

\begin{theorem}[Chow--Rashevskii]\label{T_CR}
  Let $\Omega\subseteq\R^n$ be an open connected subset, and assume
  the smooth bounded vector fields $U_1,\ldots,U_d$ satisfy the
  \emph{H\"ormander condition},
  \[ {\rm Lie}(U_1,\ldots,U_d)(x) = \R^n,
  \]
  at every point $x\in\Omega$. Then \emph{controllability} holds: for
  every input/output pair $(x,y)\in\Omega\times\Omega$, there exist
  smooth scalar controls $u^1_t,\ldots,u^d_t$ that achieve $X_1=y$,
  where $X_t$ is the solution of
  \[
    \frac{d}{dt}X_t = u^1_tU_1(X_t) + \cdots + u^d_tU_d(X_t), \quad
    X_0 = x.
  \]
\end{theorem}

\begin{example}\label{brackets_flows}
  To see why Lie brackets are relevant for controllability, it is
  useful to consider the case of linear vector fields $U(x)=Ax$ and
  $V(x)=Bx$. A particle starting at $x$ and flowing along the vector
  field $U$ for an amount of time $t$ ends up at $e^{At}x$, where the
  standard matrix exponential is used. This is because $e^{At}x$ is
  the solution of $\frac{d}{dt}X_t = AX_t$, $X_0=x$. Alternating
  between $V$, $U$, $-V$, and $-U$, therefore moves the particle from
  $x$ to $e^{-At}e^{-Bt}e^{At}e^{Bt}x$. A Taylor expansion in $t$
  shows that
  \[
    e^{-At}e^{-Bt}e^{At}e^{Bt}x = x + t^2 (AB-BA)x + O(t^3) \, .
  \]
  Therefore if $t$ is small, the alternating behavior produces motion
  in the direction $(AB-BA)x=[U,V](x)$. For general vector fields, an
  analogous computation gives the same result. The Chow--Rashevskii
  theorem is now quite intuitive: controllability holds if at each
  point one can produce motion in all directions. However, moving in
  the Lie bracket direction requires more ``energy'' (larger and more
  oscillatory controls), reflected by the short-time asymptotic $t^2$.
\end{example}

\begin{example}
  To see that a small number of vector fields can generate very large
  Lie algebras, consider the two vector fields $U(x)=x^2$ and
  $V(x)=x^k$ on $\R$, where $k\in\N$. Note that vector fields on $\R$
  are just scalar functions. Then
  $[U,V](x)=V'(x)U(x)-U'(x)V(x)=(k-2)x^{k+1}$. As a result, the Lie
  algebra generated by $x^2$ and $x^3$ contains all $x^k$, $k\ge2$.
\end{example}

In the context of deep learning, one can view the Lie bracket
operation as a way to generate \emph{features}. This requires a large
number of layers when brackets are iterated. Indeed, each layer is
associated with an Euler step of the discretized
CODE. Example~\ref{brackets_flows} then shows that four layers are
needed to move along the length-2 bracket $[U,V]$. The number of
layers required to move along a general length-$n$ bracket is
exponential in $n$.

On the other hand, the dimensionality of the feature space generated
in this way can also grow extremely quickly due to non-commutativity
of Lie brackets. Let us illustrate this using the \emph{free Lie
  algebra} on $d$ generators $Y_1,\ldots,Y_d$. This is an abstract Lie
algebra whose elements are formal linear combinations of {\em Lie
  words} in the generators. A Lie word is a formal expression
involving the generators and the bracket $[\fdot,\fdot]$, for example
$[Y_1,[Y_2,Y_1]]$ and $[Y_2,[Y_1,Y_1]]$. Two Lie words are considered
equal if they can be transformed into one another using the axioms
satisfied by the bracket, namely bilinearity, anticommutativity, and
the Jacobi identity. For example, $[Y_2,[Y_1,Y_1]]=[Y_2,0]=0$. The
dimension of the subspace $\Lcal_n$ spanned by all Lie words of length
$n$ is given by \emph{Witt's dimension formula},
\[
  \dim\Lcal_n = \frac{1}{n} \sum_{k|n}\mu(k)d^{n/k},
\]
see \cite{mag_kar_sol_76}, Theorem~5.11. Here the sum ranges over all
$k$ that divide $n$, and $\mu(\fdot)$ is the \emph{M\"obius function}
which takes values in $\{-1,0,1\}$. The asymptotic behavior for large
$n$ is exponential,
\[
  \dim\Lcal_n \sim d^n.
\]
This is related to the fact that the Lie bracket is
non-commutative. For comparison, the space of polynomials of degree at
most $n$ in $d$ commuting variables has dimension
${n+d \choose d}\sim n^d$, which only grows polynomially in $n$.

If $U_1,\ldots,U_d$ are smooth vector fields on $\R^n$ that are
sufficiently unstructured or ``generic'', we expect the Lie algebra
that they generate to behave similarly to the free Lie algebra on $d$
generators. In particular, we expect the dimensionality of the feature
space to grow very quickly. Notice, however, that the price to pay is
exponentially growing depth to generate all brackets.

\section{Universal $N$-point interpolators exist}\label{S_proof_T1}

In this section we apply the Chow-Rashevskii theorem and algebraic
results on polynomial vector fields to prove Theorem~\ref{T1}.  The
proof is constructive and relies on 5 specific linear and quadratic
vector fields $V_1 , \ldots, V_5$ for which \eqref{CODE_2} is a
universal $N$-point interpolator.

We select an arbitrary $N$ and work on the set
$\overline\Omega \subset (\R^m)^N $ of pairwise distinct $N$-tuples
$(x_1,\ldots,x_N)$ of points in $\Omega$. Here $m\ge2$ is the ambient
dimension and $N$ represents the number of training pairs as in
Section~\ref{S_gen_int}. In other words, we consider the bounded open
connected subset
\[
  \overline\Omega=\Omega^N\setminus\Delta
\]
of $(\R^m)^N$, where
\[
  \Delta = \{(x_1,\ldots,x_N)\in\Omega^N\colon \text{$x_i=x_j$ for
    some $i\ne j$}\}.
\]
($\overline\Omega$ is connected because $m\ge2$.) Then, given $d$
smooth bounded vector fields $V_1,\ldots,V_d$ on $\R^m$,
\eqref{CODE_2} is a universal $N$-point interpolator in $\Omega$ if
and only if the ``stacked'' system
\[
  \frac{d}{dt}\begin{pmatrix}X^{x_1}_t\\ \vdots\\
    X^{x_N}_t\end{pmatrix} = u^1_t \begin{pmatrix}V_1(X^{x_1}_t) \\
    \vdots \\ V_1(X^{x_N}_t) \end{pmatrix} + \cdots +
  u^d_t \begin{pmatrix}V_d(X^{x_1}_t) \\ \vdots \\
    V_d(X^{x_N}_t) \end{pmatrix}
\]
can bring any initial point
$\bar x=(x_1,\ldots,x_N)\in\overline\Omega$ to any target
$\bar y=(y_1,\ldots,y_N)\in\overline\Omega$ by means of a suitable
choice of controls $u^1_t,\ldots,u^d_t$. By the Chow--Rashevskii
theorem, this holds if and only if the stacked vector fields
\[
  V_i^{\oplus N}(\bar x) := \begin{pmatrix}V_i(x_1) \\ \vdots \\
    V_i(x_N) \end{pmatrix}, \quad i=1,\ldots,d,
\]
satisfy the H\"ormander condition at every
$\bar x=(x_1,\ldots,x_N)\in\overline\Omega$. The following definition
and subsequent lemma strongly hint at how we plan to verify the
H\"ormander condition.

\begin{definition}\label{D_inter_tuple}
  A collection $\Vcal$ of vector fields on $\R^m$ is said to {\em
    interpolate at a tuple $(x_1,\ldots,x_N)\in\overline\Omega$} if
  for every $(v_1,\ldots,v_N)\in(\R^m)^N$ there exists a vector field
  $\widehat V\in\Vcal$ such that $\widehat V(x_i)=v_i$ for all
  $i=1,\ldots,N$.
\end{definition}

\begin{lemma}\label{L_interp_Horm}
  Let $V_1,\ldots,V_d$ be smooth vector fields on $\R^m$ such that
  ${\rm Lie}(V_1,\ldots,V_d)$ interpolates at the tuple
  $\bar x=(x_1,\ldots,x_N)\in\overline\Omega$. Then
  \[ {\rm Lie}(V_1^{\oplus N},\ldots,V_d^{\oplus N})(\bar x) =
    (\R^m)^N,
  \]
  that is, the vector fields $V_1^{\oplus N},\ldots,V_d^{\oplus N}$
  satisfy the H\"ormander condition at $\bar x$.
\end{lemma}

\begin{proof}
  Pick any $\bar v=(v_1,\ldots,v_N)\in(\R^m)^N$. Since
  ${\rm Lie}(V_1,\ldots,V_d)$ interpolates at $\bar x$, it contains a
  vector field $\widehat V$ such that
  \[
    \widehat V^{\oplus N}(\bar x) = \begin{pmatrix}\widehat V(x_1) \\
      \vdots \\ \widehat V(x_N) \end{pmatrix} = \begin{pmatrix}v_1 \\
      \vdots \\ v_N \end{pmatrix}.
  \]
  Moreover, due to the identity
  $[V^{\oplus N},W^{\oplus N}]=[V,W]^{\oplus N}$, which is valid for
  any smooth vector fields $V,W$ on $\R^m$, it follows that
  ${\rm Lie}(V_1^{\oplus N},\ldots,V_d^{\oplus N})$ contains
  $\widehat V^{\oplus N}$. Therefore
  $\bar v\in{\rm Lie}(V_1^{\oplus N},\ldots,V_d^{\oplus N})(\bar x)$,
  which completes the proof.
\end{proof}

We now confirm that the collection of all polynomial vector fields
interpolates any number of pairwise distinct points.

\begin{lemma}\label{L_poly_interp}
  The set of all polynomial vector fields on $\R^m$ interpolates at
  every tuple $(x_1,\ldots,x_N)\in\overline\Omega$.
\end{lemma}

\begin{proof}
  The result follows by standard multivariate polynomial
  interpolation. Specifically, consider arbitrary
  $(x_1,\ldots,x_N)\in\overline\Omega$ and
  $(v_1,\ldots,v_N)\in(\R^m)^N$. Since the $x_i$ are pairwise
  distinct, it is possible to find, for each $j=1,\ldots,m$, a
  polynomial $p^j(x)$ on $\R^m$ such that $p^j(x_i) = v_i^j$ for
  $i=1,\ldots,N$. The vector field
  \[
    \widehat V(x) = \begin{pmatrix}p^1(x) \\ \vdots \\
      p^m(x) \end{pmatrix}
  \]
  is then polynomial and satisfies $\widehat V(x_i)=v_i$ for all
  $i=1,\ldots,N$.
\end{proof}

Thanks to the Chow--Rashevskii Theorem as stated in
Theorem~\ref{T_CR}, as well as Lemma~\ref{L_interp_Horm}
and~\ref{L_poly_interp}, in order to prove Theorem~\ref{T1} it only
remains to exhibit five smooth vector fields that do not depend on
$N$, and whose Lie algebra contains all polynomial vector fields.
This is accomplished by the following result, which therefore
completes the proof of the theorem. (Note that we actually want
\emph{bounded} vector fields. This is easily achieved by multiplying
the vector fields below by a smooth compactly supported function
$\varphi(x)$ that equals one on $\Omega$.)

\begin{proposition}\label{P_five_Vs}
  There exist $d=5$ smooth vector fields $V_1,\ldots,V_5$ on $\R^m$
  with the property that ${\rm Lie}(V_1,\ldots,V_5)$ contains all
  polynomial vector fields. Specifically, one can take
  \[
    V_1(x) = Ax, \quad V_2(x) = Bx,
  \]
  \[
    V_3(x) = \begin{pmatrix}0 \\ \vdots \\ 0 \\ 1\end{pmatrix}, \quad
    V_4(x) = \begin{pmatrix}(x^m)^2 \\ 0 \\ \vdots \\ 0\end{pmatrix},
    \quad V_5(x) = \begin{pmatrix}x^1 x^m \\ x^2x^m \\ \vdots \\
      (x^m)^2 \end{pmatrix}
  \]
  where $A$ and $B$ are suitable traceless $m\times m$ matrices, and
  $x^1,\ldots,x^m$ denote the components of $x$.\footnote{A traceless
    matrix is one whose trace is equal to zero.}
\end{proposition}

\begin{proof}
  We divide the proof into three separate statements, that together
  imply the claimed result. We use $e_1,\ldots,e_m$ to denote the
  canonical basis vectors in $\R^m$.

  \underline{Claim 1:} There is a choice of traceless $m\times m$
  matrices $A$ and $B$ such that
  ${\rm Lie}(V_1,V_2)=\{Cx\colon \text{$C$ is traceless}\}$.

  Indeed, Example~\ref{Ex_Lie_linear} shows that $\{Cx\colon \text{$C$
    is traceless}\}$ is a Lie algebra of vector fields that can be
  identified with the Lie algebra of all traceless $m\times m$
  matrices. The latter is the {\em special linear Lie algebra}
  $\mathfrak{sl}_m(\R)$, which is known to admit two generators $A$
  and $B$; see for instance \cite{kur_51}, where it is shown that in
  fact any semi-simple Lie algebra admits two generators.

  \underline{Claim 2:} With $A$ and $B$ as above,
  ${\rm Lie}(V_1,V_2,V_3,V_4)$ contains all linear vector fields.

  Indeed, we know it contains all vector fields $Cx$ with $C$
  traceless. Moreover, it contains the Lie bracket
  $[V_3,V_4](x)=2x^m e_1=2e_1e_m^\top x$. Expressing the identity
  matrix $I = (I-m e_1e_m^\top) + m e_1e_m^\top$ as a sum of a
  traceless matrix and a multiple of $2e_1e_m^\top$, it follows that
  the identity vector field $W(x)=x$ is in
  ${\rm Lie}(V_1,V_2,V_3,V_4)$. This proves the claim, since any
  matrix can be expressed as a traceless matrix plus a multiple of the
  identity.

  \underline{Claim 3:} $V_3$, $V_4$, and $V_5$ together with all
  linear vector fields generate all polynomial vector fields.

  This is asserted without proof by \cite{lei_pol_97}, and can be
  verified by direct computation. We do this in full detail in the
  appendix.

  Combining Claim~2 and Claim~3 proves the proposition.
\end{proof}

\begin{remark}
  The use of polynomials in the above proof is due to their relatively
  tractable structure. We believe the conclusion remains true for
  other classes of vector fields, also on curved spaces. For example,
  on the torus a natural choice would be to consider Fourier basis
  functions.
\end{remark}

\section{Generic expressiveness}\label{S_genexp}

Theorem~\ref{T1} shows that universal interpolators can be constructed
using just five vector fields, but not how common or rare such vector
fields are. Our next goal is to prove that parsimonious yet expressive
systems exist in great abundance. To do so, rather than using
\eqref{CODE_2} to interpolate the outputs $y_i$ directly, we will use
it to interpolate the transformed outputs $x_i+\lambda^{-1}y_i$, where
$\lambda>0$ is a (trained) constant.  Thus the input $x$ and output
$y$ are related by
\begin{equation}\label{res_readout}
  y=\lambda (X^x_1 - x),
\end{equation}
where the right-hand side can be interpreted as a particular readout
map.  Our next result shows that with five or more appropriately
\emph{randomly chosen} nonlinear vector fields, the system
\eqref{CODE_2} \& \eqref{res_readout} is sufficiently expressive to
interpolate almost every training set.

The setup of the theorem is as follows. Fix a dimension $m\ge2$ and a
bounded open connected subset $\Omega\subset\R^m$. Consider $d\ge5$
vector fields $V_1,\ldots,V_d$ that depend on a parameter $z\in\R^l$
for some $l\in\N$, in addition to their dependence on the point
$x\in\R^m$. More precisely, we assume that the components of the $V_i$
are of the form
\[
  V_i^j(x) = V_i^j(x,z), \quad i=1,\ldots,d,\quad j=1,\ldots,m, \quad
  x\in\Omega, \quad z\in\R^l,
\]
where each map $(x,z)\mapsto V_i^j(x,z)$ is real analytic in a
neighborhood of $\text{cl}(\Omega)\times\R^l$, with
$\text{cl}(\Omega)$ denoting the closure of $\Omega$.\footnote{To
  ensure that the vector fields are globally bounded on $\R^m$ for
  each fixed $z$, we multiply the given real analytic functions by a
  compactly supported function $\varphi(x)$ that equals one on
  $\Omega$. This ensures global existence and uniqueness of solutions
  to \eqref{CODE_2}. The form of the vector fields outside $\Omega$
  does not matter for the theorem.}  The vector fields are now chosen
randomly by replacing the parameter $z$ by a random vector $Z$ in
$\R^l$. We thus consider the randomly chosen vector fields
$V_i=V_i(\fdot,Z)$, $i=1,\ldots,d$. We can now state our main theorem.

\begin{theorem}\label{T2}
  Assume that
  \begin{enumerate}
  \item\label{T2_1} the law of $Z$ admits a probability density on
    $\R^l$,
  \item\label{T2_2} for some $\widehat z\in\R^l$, the Lie algebra
    generated by the $d$ vector fields
    $\widehat V_i=V_i(\fdot,\widehat z)$ corresponding to $\widehat z$
    contains all polynomial vector fields.
  \end{enumerate}
  Then with probability one, \eqref{CODE_2} \& \eqref{res_readout}
  form a universal interpolator for generic training data in the
  following sense. Consider a training set
  $\{(x_i,y_i)\in\Omega\times\Omega\colon i=1,\ldots,N\}$ of arbitrary
  size, where $(x_1,\ldots,x_N)$ is drawn from an arbitrary density on
  $(\R^m)^N$ and the $y_i$ are pairwise distinct but otherwise
  arbitrary. Then, with probability one, there exist controls
  $u^1_t,\ldots,u^d_t$ and a constant $\lambda>0$ such that
  $y_i=\lambda (X^{x_i}_1-x_i)$ for all $i$.
\end{theorem}

\begin{example}\label{E_T2}
  Fix $k\ge2$ and $d\ge5$. In order to specify $d$ polynomial vector
  fields of degree at most $k$, one needs $l=dm{m+k\choose m}$ real
  coefficients. Let $\R^l$ be the space of all such sets of
  coefficients, and let $V_i(\fdot,z)$, $i=1,\ldots,d$, be the
  polynomial vector fields specified by $z\in\R^l$. Then
  $(x,z)\mapsto V_i^j(x,z)$ is a polynomial, and in particular real
  analytic. By letting $\widehat z\in\R^l$ be the coefficients of the
  vector fields in Proposition~\ref{P_five_Vs}, we see that condition
  \ref{T2_2} of the theorem holds. Condition \ref{T2_1} holds whenever
  $Z$ is drawn from an arbitrary density on $\R^l$.
\end{example}

The proof of Theorem~\ref{T2} is presented in
Section~\ref{S_proof_T2}. Ultimately it is based on the fact that any
real analytic function is either identically zero, or nonzero on a set
of full Lebesgue measure.  Condition~\ref{T2_2} is used to exclude the
former possibility, while condition~\ref{T2_1} is used to avoid zeros
which can exist, but only constitute a nullset.

The central message of Theorem~\ref{T2} is this. The seemingly strong
property of universal interpolation is not only achieved in a
dimension-free manner as shown in Theorem~\ref{T1}. It is actually a
\emph{generic} property in the class of real analytic vector
fields. Specifically, by drawing the vector fields randomly in the
described manner, one is guaranteed with probability one that the
resulting vector fields produce a universal interpolator (at least for
generic training data and allowing for the additional trained readout
parameter $\lambda$). Possible sampling schemes include nondegenerate
normal distributions and uniform distributions on bounded open regions
of the parameter space $\R^l$.  The theorem is however more general
than that, and we make use of this in Corollary~\ref{C2} below.

The $\lambda$-scaling in \eqref{res_readout} is reminiscent of batch
normalization, especially if we were to use different parameters
$\lambda$ for different coordinates. Our mathematical results do not
require this, however. Moreover, thanks to the normalization it is not
a restriction to work with a bounded set $\Omega$.

In practice, the CODE \eqref{CODE_2} is replaced by a discretization,
say with $M$ steps. This yields a network of depth $M$. After randomly
choosing $d$ vector fields, the number of trainable parameters
(including $\lambda$ in \eqref{res_readout}) becomes $Md+1$. This
tends to be much smaller than the total number of parameters needed to
specify the vector fields, and can potentially simplify the training
task significantly. The required depth $M$ depends on the desired
training error. The fact that most parameters are chosen randomly
reinforces the view that randomness is a crucial ingredient for
training. Investigating different sampling schemes and training
algorithms in this setting is an important research question, that
will be treated elsewhere.

The fact that the sampling density for the vector field coefficients
can be completely arbitrary leads to the following simple proof that
the universal interpolator property is in a certain sense generic in
the class of all smooth vector fields.

\begin{corollary}\label{C1}
  Fix $m\ge2$ and a bounded open connected subset
  $\Omega\subset\R^m$. Consider $d\ge5$ smooth vector fields
  $\widehat V_1,\ldots,\widehat V_d$, and a tolerance
  $\varepsilon>0$. Then there exist smooth vector fields
  $V_1,\dots,V_d$ that are uniformly $\varepsilon$-close to the given
  vector fields on $\Omega$, in the sense that
  \[
    \sup_{x\in\Omega}\| V_i(x) - \widehat V_i(x)\| < \varepsilon,
    \quad i=1,\ldots,d,
  \]
  and such that \eqref{CODE_2} \& \eqref{res_readout} form a universal
  interpolator for generic training data in the sense of
  Theorem~\ref{T2}.
\end{corollary}

\begin{proof}
  By polynomial approximation, there exist polynomial vector fields
  $W_1,\ldots,W_d$ with
  $\sup_{x\in\Omega}\| W_i(x) - \widehat V_i(x)\|<\varepsilon/2$ for
  all $i$. Let $k$ be the largest degree among the $W_i$, but no
  smaller than $2$. Parameterize all polynomial vector fields of
  degree at most $k$ by a coefficient vector $z\in\R^l$ as in
  Example~\ref{E_T2}.  Let $\Theta\subset\R^l$ be the set of all
  coefficients corresponding to polynomial vector fields
  $V_1,\ldots,V_d$ with
  $\sup_{x\in\Omega}\| W_i(x) - V_i(x)\|<\varepsilon/2$ for all
  $i$. Then $\Theta$ is an open set, so we can find a probability
  density concentrated on $\Theta$. Thanks to Theorem~\ref{T2} and
  Example~\ref{E_T2}, by drawing a coefficient vector $Z$ from this
  density we get, with probability one, vector fields $V_1,\ldots,V_d$
  with the required properties.
\end{proof}

Our second corollary establishes a randomly chosen set of neural
network type vector fields that satisfy the universal interpolator
property.

\begin{corollary}\label{C2}
  Fix $m\ge2$ and a bounded open connected subset
  $\Omega\subset\R^m$. Consider $d=7$ vector fields of the form
  \[
    V_i(x) = \sigma_i(C_i x + b_i ), \quad i=1,\ldots,7,
  \]
  where each $C_i$ is a random matrix in $\R^{m\times m}$, $b_i$ a
  random vector in $\R^m$, and $\sigma_i(\fdot)$ a real analytic
  nonlinearity acting componentwise, parameterized by some random
  vector $Z_0$ in a real analytic manner.  Assume that for some value
  $\widehat z_0$ of $Z_0$, we have $\sigma_i(r)=r$ for $i=1,2,3$, and
  $\sigma_i(r)=r^2$ for $i=4,5,6,7$. Assume also that the random
  elements $Z_0,C_1,\ldots,C_7,b_1,\ldots,b_7$ admit a joint
  density. Then with probability one, \eqref{CODE_2} \&
  \eqref{res_readout} form a universal interpolator for generic
  training data in the sense of Theorem~\ref{T2}.
\end{corollary}

\begin{proof}
  To apply Theorem~\ref{T2}, first observe that the vector fields
  $V_1,\ldots,V_7$ are jointly real analytic in $x$ and in the random
  vector $Z$ consisting of $Z_0,C_1,\ldots,C_7,b_1,\ldots,b_7$.  This
  admits a density by assumption, so condition~\ref{T2_1} of the
  theorem is satisfied. It only remains to verify
  condition~\ref{T2_2}. Define the vector fields
  $\widehat V_1(x) = Ax$ and $\widehat V_2(x) = Bx$, where $A$ and $B$
  are the traceless $m\times m$ matrices from
  Proposition~\ref{P_five_Vs}. Define also the vector fields
  \[
    \widehat V_3(x) = \begin{pmatrix}0 \\ \vdots \\ 0 \\
      1\end{pmatrix}, \quad \widehat V_4(x) = \begin{pmatrix}(x^m)^2
      \\ 0 \\ \vdots \\ 0\end{pmatrix},
    \quad \widehat V_5(x) = \begin{pmatrix}(x^1 + x^m)^2 \\ (x^2 + x^m)^2 \\ \vdots \\
      (x^m + x^m)^2 \end{pmatrix}
  \]
  \[
    \widehat V_6(x) = \begin{pmatrix}(x^1)^2 \\ (x^2)^2 \\ \vdots \\
      (x^m)^2 \end{pmatrix} , \quad
    \widehat V_7(x) = \begin{pmatrix}(x^m)^2 \\ (x^m)^2 \\ \vdots \\
      (x^m)^2 \end{pmatrix}.
  \]
  Then the five vector fields $\widehat V_1,\ldots,\widehat V_4$, and
  $\frac{1}{2}(\widehat V_5 - \widehat V_6 - \widehat V_7)$ are
  exactly the ones from Proposition~\ref{P_five_Vs}. The Lie algebra
  they generate, and therefore also the Lie algebra generated by
  $\widehat V_1,\ldots,\widehat V_7$, contains all polynomial vector
  fields.  Let now $\widehat z$ be the value of $Z$ for which
  $Z_0=\widehat z_0$, $b_1=b_2=b_4=b_5=b_6=b_7=0$, $b_3=e_m$ (the
  $m$th canonical basis vector), $C_1=A$, $C_2=B$, $C_3=0$,
  \[
    C_4 = \begin{pmatrix}0 & \cdots & 0 & 1\\ 0 & \cdots &0&0\\\vdots
      &&\vdots& \vdots\\0 &\cdots&0& 0\end{pmatrix},\quad C_7
    = \begin{pmatrix}0 & \cdots & 0 & 1\\ 0 & \cdots &0&1\\\vdots
      &&\vdots& \vdots\\0 &\cdots&0& 1\end{pmatrix},
  \]
  $C_6=I$ (the $m\times m$ identity matrix), and $C_5=C_6+C_7$. For
  this value $\widehat z$ of $Z$, the vector fields $V_1,\ldots,V_7$
  coincide with $\widehat V_1,\ldots,\widehat
  V_7$. Condition~\ref{T2_2} of Theorem~\ref{T2} is therefore
  satisfied, and the proof is complete.
\end{proof}

\begin{example}
  We illustrate Corollary~\ref{C2} with one concrete example. Let all
  the entries of the matrices $C_i$ and vectors $b_i$ be standard
  normal.  Choose a fixed real analytic nonlinearity $\sigma(\fdot)$,
  for example $\sigma(r)=\arctan(r)$ or $\sigma(r)=\tanh(r)$. Define
  \[
    \sigma_i(r) = Z_0^1 r + (1-Z_0^1) \sigma(r)
  \]
  for $ i =1, 2, 3$, and
  \[
    \sigma_i(r) = Z_0^2 r^2 + (1 - Z_0^2) \sigma(r)
  \]
  for $ i = 4,5,6,7 $, with the two components of $Z_0=(Z_0^1,Z_0^2)$
  standard normal. All random variables are taken mutually
  independent. The hypotheses of the corollary are then satisfied with
  $\widehat z_0=(1,1)$.
\end{example}

\section{Proof of Theorem~\ref{T2}}\label{S_proof_T2}

We focus on the case $d=5$. The result for larger values of $d$ then
follows by restricting to controls in the CODE \eqref{CODE_2} with
$u^6_t=\ldots=u^d_t=0$. (Of course, more than five vector fields could
still be important to achieve better results in practice.)

Consider therefore vector fields $V_1(\fdot,z),\ldots,V_5(\fdot,z)$ on
$\R^m$, parameterized by a parameter $z\in\R^l$, such that the map
$(x,z)\mapsto V_i^j(x,z)$ is real analytic in a neighborhood of
$\text{cl}(\Omega)\times\R^l$ for all $i=1,\ldots,d$ and
$j=1,\ldots,m$. Recall from condition~\ref{T2_2} of the theorem that
$\widehat V_i=V_i(\fdot,\widehat z)$ denote the vector fields obtained
by taking $z=\widehat z$ which, by assumption, has the property that
${\rm Lie}(\widehat V_1,\ldots,\widehat V_5)$ contains all polynomial
vector fields.

The following lemma is the technical core of the proof of
Theorem~\ref{T2}. It uses the notion of \emph{interpolating at a
  tuple}, introduced in Definition~\ref{D_inter_tuple}.

\begin{lemma}\label{L_pfT2}
  Fix any $N\in\N$. There exists a Lebesgue nullset
  $\Mcal_N\subset \R^l $ with the following property: for every
  $z\in\R^l \setminus \Mcal_N$, there exists a Lebesgue nullset
  $\Ncal_N\subset\Omega^N$ (depending on $z$) such that
  ${\rm Lie}(V_1,\ldots,V_5)$ interpolates at every tuple
  $\bar x=(x_1,\ldots,x_N)\in\Omega^N\setminus\Ncal_N$.
\end{lemma}

\begin{proof}
  For each $n\in\N$, let $D_n=m{m+n\choose m}$ denote the dimension of
  the space of polynomial vector fields on $\R^m$ of degree at most
  $n$.\footnote{$D_n$ is $m$ times ${m+n\choose m}$, the dimension of
    the space of polynomials of degree at most $n$ in $m$ variables.}
  Since ${\rm Lie}(\widehat V_1,\ldots,\widehat V_5)$ contains all
  polynomial vector fields, it contains in particular a sequence of
  vector fields $E_1,E_2,\ldots$ such that for each $n$,
  $\{E_1,\ldots,E_{D_n}\}$ forms a basis for the space of polynomial
  vector fields of degree at most $n$. By definition of the Lie
  algebra, each $E_j$ is of the form
  \[
    E_j(x) = L_j(\widehat V_1,\ldots,\widehat V_5)(x)
  \]
  for some Lie polynomial $L_j$ on five symbols (i.e., a linear
  combination of Lie words built from iterated brackets).

  Consider now an arbitrary $z\in\R^l$ and the corresponding vector
  fields $V_i=V_i(\fdot,z)$, $i=1,\ldots,5$. For each $n\in\N$, define
  the collection of vector fields
  \[
    \Vcal_n = \text{linear span of
      $L_1(V_1,\ldots,V_5),\ldots,L_{D_n}(V_1,\ldots,V_5)$}.
  \]
  The collection $\Vcal_n$ interpolates at a tuple
  $(x_1,\ldots,x_N)\in\Omega^N$, in the sense of
  Definition~\ref{D_inter_tuple}, if and only if the $mN \times D_n$
  matrix
  \[
    \begin{pmatrix}
      L_1(V_1,\ldots,V_5)(x_1)	& \cdots 	& L_{D_n}(V_1,\ldots,V_5)(x_1) \\
      \vdots				&		& \vdots\\
      L_1(V_1,\ldots,V_5)(x_N) & \cdots & L_{D_n}(V_1,\ldots,V_5)(x_N)
    \end{pmatrix}
  \]
  has columns that span $(\R^m)^N$. This holds if and only if at least
  one $mN\times mN$ submatrix has nonzero determinant, which in turn
  holds if and only if the nonnegative quantity
  \[
    \Gamma_n = \sum_{\substack{J\subseteq\{1,\ldots,D_n\}\\|J|=mN}}
    \det\begin{bmatrix*}\begin{pmatrix*}[l]L_j(V_1,\ldots,V_5)(x_1) \\
        \vdots \\ L_j(V_1,\ldots,V_5)(x_N)\end{pmatrix*} ,\ j\in
      J\end{bmatrix*}^2
  \]
  is strictly positive.

  Since products, sums, and derivatives of real analytic functions
  remain real analytic, we have that
  $\Gamma_n = \Gamma_n(x_1,\ldots,x_N, z)$ is jointly real analytic in
  $(x_1,\ldots,x_N,z)$ in a neighborhood of
  $(\text{cl}(\Omega))^N\times\R^l$.  Furthermore, by construction,
  the vector fields $L_j(\widehat V_1,\ldots,\widehat V_5)$,
  $j=1,\ldots,D_n$, span all polynomial vector fields of degree at
  most $n$. Therefore, in view of Lemma~\ref{L_poly_interp}, for $n$
  large enough depending on $N$, we have
  \[
    \Gamma_n(x_1,\ldots,x_N,\widehat z) > 0
  \]
  for all pairwise distinct $(x_1,\ldots,x_N)\in(\R^m)^N$. In
  particular,
  \[
    (x_1,\ldots,x_N,z) \mapsto \Gamma_n(x_1,\ldots,x_N,z)
  \]
  is not identically zero and thus, being a nonnegative real analytic
  function, is strictly positive almost everywhere.  Therefore, there
  is a Lebesgue nullset $\Mcal_N\subset \R^l $ such that whenever
  $z\in \R^l \setminus \Mcal_N$, the real analytic function
  \[
    (x_1,\ldots,x_N) \mapsto \Gamma_n(x_1,\ldots,x_N,z)
  \]
  is not identically zero. Its zero set,
  \[
    \Ncal_N =
    \{(x_1,\ldots,x_N)\in\Omega^N\colon\Gamma_n(x_1,\ldots,x_N,z)=0\},
  \]
  is then a Lebesgue nullset. (Note that $\Ncal_N$ depends on the
  choice of $z$.)  Since ${\rm Lie}(V_1,\ldots,V_5)$ contains
  $\{L_j(V_1,\ldots,V_5)\colon j=1,\ldots,D_n\}$, and hence contains
  $\Vcal_n$ as well, it interpolates at every tuple
  $\bar x=(x_1,\ldots,x_N)\in\Omega^N\setminus\Ncal_N$. The lemma is
  proved.
\end{proof}

We can now prove Theorem~\ref{T2}. Let $\Mcal_N\subset \R^l $ for
$N\in\N$ be the nullsets given in Lemma~\ref{L_pfT2}. Define
\[
  \Mcal = \bigcup_{N=1}^\infty\Mcal_N,
\]
which is still a nullset.  Assume now that $V_i=V_i(\fdot,Z)$,
$i=1,\ldots,d$, are chosen randomly as described in the theorem. Then,
since the law of $Z$ has a density, $Z\in\R^l \setminus\Mcal$ with
probability one.  Fix any $N\in\N$ and let $\Ncal_N\subset\Omega^N$ be
the nullset whose existence is guaranteed by
Lemma~\ref{L_pfT2}. Choose
$\{(x_i,y_i)\in\Omega\times\Omega\colon i=1,\ldots,N\}$ as described
in the theorem. Then $\bar x=(x_1,\ldots,x_N)$ lies in
$\Omega^N\setminus\Ncal_N$ with probability one, so that
${\rm Lie}(V_1,\ldots,V_5)$ interpolates at $\bar
x$. Lemma~\ref{L_interp_Horm} now implies that the H\"ormander
condition holds at $\bar x$:
\[ {\rm Lie}(V_1^{\oplus N},\ldots,V_d^{\oplus N})(\bar x) = (\R^m)^N.
\]
By continuity, there is an open connected neighborhood
$\Ucal\subset\Omega^N$ of $\bar x$ such that the H\"ormander condition
holds everywhere in $\Ucal$. Moreover, since $\Ucal$ is open, it is
possible to choose $\lambda>0$ large enough that
$\bar x+\lambda^{-1}\bar y\in\Ucal$, where $\bar
y=(y_1,\ldots,y_N)$. We can then apply the Chow--Rashevskii theorem in
$\Ucal$ to get controls $u^1_t,\ldots,u^5_t$ that achieve
$x_i+\lambda^{-1}y_i=X^{x_i}_1$ for $i=1,\ldots,N$. This completes the
proof of the theorem.

\begin{remark}
  We conjecture that $d=2$ vector fields would actually be sufficient
  for the conclusion of Theorem \ref{T2}.  Notice also how the
  re-scaling trick of introducing an additional parameter $\lambda$
  localizes the problem. This circumvents potentially very difficult
  questions about the global structure of the zero sets $\Ncal_N$,
  that may prevent us from applying the Chow--Rashevskii theorem
  globally.
\end{remark}

\appendix

\section{Generators for the polynomial vector fields}

In this appendix we verify Claim~3 in the proof of
Proposition~\ref{P_five_Vs}. To avoid confusion with powers, we here
use subscripts to denote the components of the vector
$x=(x_1,\ldots,x_m)$. Moreover, to make computations more transparent
we canonically identify any vector field
$V(x)=f_1(x)e_1+\cdots+f_m(x)e_m$ on $\R^m$ with the differential
operator $f_1(x)\partial_1+\cdots+f_m(x)\partial_m$, which we again
denote by $V$. Here $\partial_i=\frac{\partial}{\partial x_i}$ denotes
partial derivative with respect to $x_i$. The action of $V$ on a
smooth scalar function $g$ is
$Vg=f_1\partial_1g+\cdots+f_m\partial_mg$. The Lie bracket of two
vector fields $f\partial_i$ and $g\partial_j$ is
$[f\partial_i,g\partial_j]=f\partial_ig - g\partial_jf$.\footnote{In
  particular, this gives the formula $[U,V]g = U(Vg) - V(Ug)$ for
  every smooth function $g$, showing that the Lie bracket of vector
  fields coincides with the linear commutator of the associated
  differential operators.}

We now proceed with the proof. Let $\Lcal$ be the Lie algebra
generated by the vector fields
\[
  \partial_m,\quad x_m^2\partial_1, \quad x_m\sum_{i=1}^m
  x_i\partial_i, \quad x_i\partial_j \quad (i,j=1,\ldots,m).
\]
We must show that $\Lcal$ contains all polynomial vector fields.

Let us first show that $\Lcal$ contains all polynomial vector fields
of degree at most two. All linear vector fields lie in $\Lcal$ by
assumption. Furthermore, all constant vector fields lie in $\Lcal$
because $\partial_m\in\Lcal$ by assumption, and
$\partial_i=[\partial_m,x_m\partial_i]\in\Lcal$ for $i=1,\ldots,m-1$.

We now turn to the quadratic vector fields, and start by considering
the following identities. For $i\in\{1,\ldots,m-1\}$, we compute
\begin{equation}\label{A_calc_1}
  \begin{aligned}
    2x_{m-1}x_m\partial_i &= [x_{m-1}\partial_m,x_m^2\partial_i] \\
    x_{m-2}x_m\partial_i &= [x_{m-2}\partial_{m-1},x_{m-1}x_m\partial_i] \\
    &\vdots \\
    x_{i+1}x_m\partial_i &=
    [x_{i+1}\partial_{i+2},x_{i+2}x_m\partial_i],
  \end{aligned}
\end{equation}
where the last line is only included if $i\le m-2$. For
$i\in\{2,\ldots,m\}$ we compute
\begin{align*}
  2x_1x_m\partial_i &= [x_1\partial_m,x_m^2\partial_i] \\
  x_2x_m\partial_i &= [x_2\partial_1,x_1x_m\partial_i] \\
                    &\vdots \\
  x_{i-1}x_m\partial_i &= [x_{i-1}\partial_{i-2},x_{i-2}x_m\partial_i].
\end{align*}
Moreover, we have $x_m^2\partial_i=[x_m^2\partial_1,x_1\partial_i]$
for $i=1,\ldots,m-1$. From these computations we deduce that $\Lcal$
contains all vector fields of the form $f(x)\partial_i$, where
$i\in\{1,\ldots,m-1\}$ and $f(x)$ ranges across the monomials listed
in the following matrix:
\begin{equation}\label{matrix_1}
  \begin{pmatrix}
    x_1^2	&x_1x_2	&\cdots	&x_1x_{i-1}	& 0	& x_1x_{i+1}	& \cdots	& x_1x_m \\
    &x_2^2	&\cdots	&x_2x_{i-1}	& 0	& x_2x_{i+1}	& \cdots	& x_2x_m \\
    &		&\ddots	&\vdots		&	&\vdots		&		& \vdots \\
    &		&		&x_{i-1}^2		& 0	&x_{i-1}x_{i+1}	& \cdots	& x_{i-1}x_m \\
    &		&		&			& 0	&0			& \cdots	& 0\\
    &		&		&			& 	&x_{i+1}^2	& \cdots	& x_{i+1}x_m \\
    &		&		&			&	&			& \ddots	& \vdots \\
    &		&		&			& 	&			& 		& x_m^2 \\
  \end{pmatrix}
\end{equation}
We now extend this to $i=m$. A calculation shows that
\[
  -x_m^2\partial_m = [x_m^2\partial_1,x_1\partial_m] + 2
  \sum_{i=1}^{m-1}[x_i\partial_{i+1},x_{i+1}x_m\partial_i],
\]
which therefore lies in $\Lcal$. Repeating \eqref{A_calc_1}, this time
with $i=m$, gives
\begin{align*}
  2x_{m-1}x_m\partial_m &= [x_{m-1}\partial_m,x_m^2\partial_m] \\
  x_{m-2}x_m\partial_m &= [x_{m-2}\partial_{m-1},x_{m-1}x_m\partial_m] \\
                        &\vdots \\
  x_1x_m\partial_m &= [x_1\partial_2,x_2x_m\partial_m].
\end{align*}
Moreover, we have $x_jx_k\partial_m=[x_j\partial_m,x_kx_m\partial_m]$
for $j,k<m$. From this we deduce that $\Lcal$ additionally contains
all vector fields of the form $f(x)\partial_m$, where $f(x)$ ranges
across the monomials listed in \eqref{matrix_1} with $i=m$.

There are still monomials missing in \eqref{matrix_1}. Consider first
the case $i=m$. We have
$2x_jx_m\partial_m=[x_j\partial_m,x_m^2\partial_m]$ for $j<m$, and
$x_m^2\partial_m\in\Lcal$ by assumption. This confirms that
$f(x)\partial_m\in\Lcal$ whenever $f(x)$ is a monomial of degree
two. Consider instead the case $i<m$. We compute
\begin{align*}
  x_i\sum_{j=1}^mx_j\partial_j &= [x_i\partial_m, x_m\sum_{j=1}^m x_j\partial_j] + x_ix_m\partial_m \\
  -x_ix_m^2\partial_m &= [x_m^2\partial_m, x_i\sum_{j=1}^mx_j\partial_j] \\
  x_ix_m\partial_i &= x_m^2\partial_m + \frac12 [[\partial_m,x_ix_m^2\partial_m],x_m\partial_i].
\end{align*}
This implies that $x_ix_m\partial_i\in\Lcal$ for all
$i<m$. Furthermore, for all $i\ne j$ we have
\begin{align*}
  x_i^2\partial_i &= [x_i\partial_m, x_ix_m\partial_i] + x_ix_m\partial_m \\
  2x_ix_j\partial_i &= [x_j\partial_i,x_i^2\partial_i].
\end{align*}
This confirms that $f(x)\partial_i\in\Lcal$ whenever $f(x)$ is a
monomial of degree two and $i\in\{1,\ldots,m-1\}$. In summary, we have
shown that $\Lcal$ contains all polynomial vector fields of degree at
most two.

It remains to prove that $\Lcal$ contains all higher-degree polynomial
vector fields as well. This follows by induction from the following
claim; note that we have already established the base case $k=2$.

\underline{Claim:} Let $k\ge2$ and assume $\Lcal$ contains all
$x^{\bm\alpha}\partial_i$ with $|\bm\alpha|\le k$. Then $\Lcal$ also
contains all $x^{\bm\alpha}\partial_i$ with $|\bm\alpha|=k+1$.

To prove the claim, pick $\bm\alpha$ with $|\bm\alpha|=k+1$. We prove
that $\Lcal$ contains $x^{\bm\alpha}\partial_1$; the vector fields
$x^{\bm\alpha}\partial_i$ with $i=2,\ldots,m$ are treated in the same
way. There are three cases. First, if $\alpha_1=0$, then
$\alpha_i\ge1$ for some $i\ge2$. Thus
$2x^{\bm\alpha}\partial_1=[x^{\alpha-e_1}\partial_i,x_i^2\partial_1]\in\Lcal$. Second,
if $\alpha_1\ge1$ and $\alpha_1\ne3$, then
$(3-\alpha_1)x^{\bm\alpha}\partial_1=[x^{\alpha-e_1}\partial_1,x_1^2\partial_1]$,
so that $x^{\bm\alpha}\partial_1\in\Lcal$. Third, if $\alpha_1=3$, we
have $x^{\bm\alpha}=x_1^3 x^{\bm\beta}$ with
$\bm\beta=(0,\alpha_2,\ldots,\alpha_m)$. Then
$2x^{\bm\alpha}\partial_1=[x_1x^{\bm\beta}\partial_1,[x_1^2\partial_2,x_1x_2\partial_1]
+ 2[x_1^2\partial_1,x_1x_2\partial_2]]\in\Lcal$. This completes the
proof of the claim, and shows that $\Lcal$ contains all polynomial
vector fields.

\bibliographystyle{abbrvnat} \bibliography{bibl}
\end{document}